\documentclass[10pt]{article}
\usepackage{latexsym, amssymb, amsmath, amsthm, amscd}
\usepackage[all,cmtip]{xy}
\usepackage{amsmath}
\usepackage{amsthm}
\usepackage{amssymb}
\usepackage{mathtools}
\usepackage{amsfonts}
\usepackage{amsrefs}
\usepackage{color,authblk}
\usepackage{tikz}
\usepackage{amscd} 
\usepackage{comment}
\usepackage{mathrsfs}
\usepackage[all]{xy}
\usepackage{stackrel}
\usepackage{graphicx}
\usepackage{authblk}
\usepackage[hidelinks]{hyperref}
\usepackage{color,soul}
\usepackage{fullpage}
\usepackage[all,cmtip]{xy}
\usepackage{amsmath,amscd}
\usepackage{tikz-cd}
\usepackage{tikz}
\usetikzlibrary{matrix}

\usepackage{hyperref}

\numberwithin{equation}{section} \DeclareMathSizes{2}{10}{12}{13}
\newtheorem{thm}{Proposition}[section]
\newtheorem{Thm}[thm]{Theorem}
\newtheorem{rem}[thm]{Remark}

\newtheorem{cor}[thm]{Corollary}
\newtheorem{lem}[thm]{Lemma}
\newtheorem{defn}[thm]{Definition}

\numberwithin{thm}{section} 

\title{Hilbert spaces over $C^*$-tensor categories, Fredholm modules and cyclic cohomology}

\author{Abhishek Banerjee\footnote{Department of Mathematics, Indian Institute of Science, Bangalore, India. Email: abhishekbanerjee1313@gmail.com}  $\qquad\qquad$ Subhajit Das \footnote{Department of Mathematics, Indian Institute of Science, Bangalore, India. Email: subhajitdas@iisc.ac.in} $\qquad\qquad$ Surjeet Kour\footnote{Department of Mathematics, Indian Institute of Technology, Delhi, India. Email: koursurjeet@gmail.com} }

\date{}

\begin{document}

\maketitle 

\medskip

\begin{abstract} 
We construct Fredholm modules over an algebra taking values in generalized Hilbert spaces over a rigid $C^*$-tensor category. Using methods of Connes, we obtain Chern characters taking values in cyclic cohomology. These Chern characters are well behaved with respect to the periodicity operator, and depend only on the homotopy class of the Fredholm module.
\end{abstract}

\medskip
MSC(2020) Subject Classification: 16E40, 18D10, 47A53, 58B34


\hypersetup{linktocpage}

\tableofcontents

\section{Introduction}

A $C^*$-tensor category is a monoidal category with certain additional structures that are motivated by the definition of a $C^*$-algebra. This includes a ``dagger structure'' on its hom spaces, as well as norms on morphisms that satisfy the $C^*$-identity. In recent years, objects such as $C^*$-tensor categories and fusion categories have been studied by a number of authors. The results obtained have provided a beautiful mix of representation theory, categorical algebra and functional analysis (see, for instance, \cite{Chen}, \cite{GLR}, \cite{GhJo}, \cite{HePe}, \cite{JP}, \cite{JL}, \cite{LR}). 

\smallskip
Our aim is to take this idea further. Our starting point in this paper is the work of Jones and Penneys \cite{JP}, which introduces a notion of generalized Hilbert space over a $C^*$-tensor category $\mathscr X$. The category of Hilbert spaces over $\mathscr X$ is denoted by  $\mathbf{Hilb}(\mathcal X)$. We combine this with the methods of Connes \cite{Con2}, \cite{Con}, \cite{Con4} in noncommutative geometry. In particular, Connes \cite{Con} used $p$-summable Fredholm modules over algebras to obtain classes in cyclic cohomology. In this paper, we introduce Fredholm modules taking values in generalized Hilbert spaces over the rigid $C^*$-tensor category  $\mathscr X$. Our main results are as follows. First, we use the methods analogous to Connes \cite{Con} to obtain ``cycles'' over algebras corresponding to  Fredholm modules constructed over $\mathbf{Hilb}(\mathcal X)$. This gives Chern characters taking values in cyclic cohomology. These characters are well behaved with respect to the periodicity operator obtained by means of taking the cup product with a certain distinguished element in cyclic cohomology. Finally, we show that the Chern character of such a Fredholm module depends only on its homotopy class. We also mention here our previous work in \cite{BB}, where we studied Fredholm modules over small preadditive categories, along with classes induced in cyclic cohomology.

\section{Preliminaries}

\smallskip
We always let $\mathscr X$ be a rigid $\rm C^*$-tensor category (see for instance, \cite[$\S$ 2.3]{JP}). More explicitly, we assume that:

\smallskip
(1) $\mathscr X$ is a tensor category, that is, an abelian linear monoidal category $(\mathscr X, \otimes, 1)$ with the hom-spaces being finite dimensional complex vector spaces. Throughout, we will assume that the monoidal structure on $\mathscr X$ is strict. 

\smallskip
(2) $\mathscr X$ has a dagger structure, that is, for every pair $a, b \in Ob(\mathscr X)$, there is an  antilinear map $(\_)^* : \mathscr X(a, b) \longrightarrow \mathscr X(b,a)$, called the adjoint. The adjoint satisfies $f^{**}=f$ for any morphism $f\in \mathscr X$ and $(g \circ f)^* = f^* \circ g^*$ for composable morphisms $f,g$ in $\mathscr X$. An isomorphism $f \in \mathscr X$ is said to be unitary if $f^{-1} = f^*$.

\smallskip
(3) The tensor structure and the dagger structure are compatible in the sense that, the associators and unitors of the monoidal structure are unitary isomorphisms and $(g \otimes f)^{*} = g^{*} \otimes f^{*}$ for all morphisms $f$ and $g$.

\smallskip
(4) The dagger structure of $\mathscr X$ is $\rm C^*$, that is, (i) For every $a, b \in \mathscr X$, $f \in \mathscr X(a,b)$, there is a $g \in \mathscr X(a, a)$ such that $f^* \circ f = g^* \circ g$ and (ii) For each $a, b \in \mathscr X$, the function $||.|| : \mathscr X(a, b) \longrightarrow [0, \infty)$ defined by $||f||^2 := sup\{|\xi| \geq 0 : f^* \circ f - \xi 1_a\text{ is not invertible}\}$ defines a complete norm on $\mathscr X(a, b)$. These norms are submultiplicative, i.e., $||g \circ f|| \leq ||g||.||f||$ for all composable morphisms $f, g$, and satisfy the $\rm C^*-$identity  $||f^* \circ f|| = ||f||^2$ for all morphisms $f \in \mathscr X$.

\smallskip
(5) The tensor structure is rigid, that is, every object $a \in \mathscr X$ has a dual object $(a^{\vee}, ev_a : a^{\vee} \otimes a \longrightarrow 1, coev_a : 1 \longrightarrow a \otimes a^{\vee})$ and an object $a_{\vee}$ (called a predual of $a$) such that $(a_{\vee})^{\vee} \cong a$.

\smallskip
We recall that an object $c \in \mathscr X$ is called simple if it has no subobjects other than $0$ and itself. It then follows by Schur's Lemma that if $c$ and $c'$ are two isomorphic (resp. non-isomorphic) simple objects in $\mathscr X$, then $\mathscr X(c, c') \cong \mathbb{C}$ (resp. $\mathscr X(c, c') \cong 0$). We also assume that the unit object $1$ of $\mathscr X$ is simple.

\smallskip
We now highlight some consequences of these assumptions :

\smallskip
(1) $\mathscr X$ is a semisimple category (see for instance, \cite{GLR}), that is, every object in $\mathscr X$ is a finite direct sum of simple objects.

\smallskip
(2) $\mathscr X$ has a canonical bi-involutive structure (see for instance, \cite[$\S$ 2.2, 2.3]{JP})
\begin{equation}\label{E2.1}
\left( \overline{\cdot} : \mathscr X \longrightarrow \mathscr X,\text{ }(\varphi_c : c \xrightarrow{\sim} \overline{\overline{c}})_{c \in \mathscr X},\text{ }(\nu_{a, b} : \overline{a} \otimes \overline{b} \xrightarrow{\sim} \overline{b \otimes a})_{a, b \in \mathscr X},\text{ }r : 1 \xrightarrow{\sim} \overline{1}\right)
\end{equation}  and $\overline{a}$ is a dual object of $a$ for any $a \in \mathscr X$. Further, we have 
\begin{equation}\label{2.3d} coev_a^* \circ (f \otimes 1_{\overline{a}}) \circ coev_{a} = ev_a \circ (1_{\overline{a}} \otimes f) \circ ev_a^*
\end{equation} for all morphisms $f \in \mathscr X(a,a)$. In particular, for any object $a \in \mathscr X$, setting $f = id_a$ in \eqref{2.3d} gives a scalar multiple $d_a\cdot id_1$ of the identity. Then $d_a\in \mathbb C$   is called the quantum dimension   of the object $a$.

\smallskip
(3) Each hom space $\mathscr X(a, b)$ has a Hilbert space structure (see \cite[Definition 2.15]{JP} given by,
\begin{equation}\label{E1}
  \langle f, g\rangle_{\mathscr X(a, b)}\text{ }:=\text{ }1 \xrightarrow{coev_a} a \otimes \overline{a} \xrightarrow{f \otimes 1_{\overline{a}}} b \otimes \overline{a} \xrightarrow{g^* \otimes 1_{\overline{a}}} a \otimes \overline{a} \xrightarrow{\varphi_a \otimes 1_{\overline{a}}} \overline{\overline{a}} \otimes \overline{a} \xrightarrow{ev_{\overline{a}}} 1
\end{equation} We note that $\mathscr X(1,1) \cong \mathbb{C}$ since $1$ is assumed to be simple. Thus, the composite morphism $1 \longrightarrow 1$ in $\eqref{E1}$ can be identified with a scalar. 

\smallskip
We now fix a set $Irr(\mathscr X)$ of representatives of isomorphism classes of simple objects in $\mathscr X$. Let $\mathbf{Vec}$ denote the category of complex vector spaces. Let $\mathbf{Vec}(\mathscr X)$ (see \cite[$\S$ 2.4]{JP}) denote the category of linear functors $\mathscr X^{op} \longrightarrow \mathbf{Vec}$. 
\smallskip
Let $\mathbf{Hilb}$ denote the dagger tensor category of  separable  complex Hilbert spaces and bounded linear operators.   
 The objects of the category $\mathbf{Hilb}(\mathscr X)$ are linear dagger functors $\mathbf  H:\mathscr X^{op} \longrightarrow \mathbf{Hilb}$ with morphisms given by setting for any
 $\mathbf  H$, $\mathbf  K \in \mathbf{Hilb}(\mathscr X)$  (see \cite[$\S$ 2.6]{JP})
 \begin{equation}\label{26f}
 \mathbf{Hilb}(\mathscr X)(\mathbf  H, \mathbf  K):=\{\mbox{$\theta\in Nat(\mathbf  H, \mathbf  K)$, $sup_{c \in \mathscr X} |\theta_c|< \infty$}\}
 \end{equation} In $\eqref{26f}$, each $|\theta_c|$ denotes the norm of the bounded linear operator $\theta_c:\mathbf {H}(c)\longrightarrow \mathbf {K}(c)$ and $Nat(\mathbf {H},\mathbf {K})$ denotes the natural transformations from $\mathbf {H}$ to $\mathbf {K}$.
\section{$p$-Schatten classes and Fredholm modules}
For Hilbert spaces $H$, $H'$, let $B(H,H')$ be the collection of bounded linear operators from $H$ to $H'$ and let $K(H,H')\subseteq B(H,H')$ be the subspace of compact operators. For an operator $T\in K(H,H')$, its $p$-Schatten norm (for $p\geq 1$) is given by
\begin{equation}\label{sct2}
||T||_p:=\left(\underset{n\geq 1}{\sum} s^p_n(T)\right)^{1/p}
\end{equation} where $s_1(T)\geq s_2(T)\geq .... $ is the sequence of singular values of $T$, i.e., the eigenvalues of the operator $\sqrt{T^*T}$ (see, for instance, ). The $p$-th Schatten class $S_p(H,H')\subseteq K(H,H')$ consists of all operators $T\in K(H,H')$ such that $||T||_p$ is finite. 

\begin{defn}\label{D3.1} Let $p\in [1,\infty)$. 
For $\mathbf {H}, \mathbf {H}' \in \mathbf{Hilb}(\mathscr X)$, we define the $p$-Schatten class $\mathcal S_p(\mathbf {H}, \mathbf {H}')\subseteq \mathbf{Hilb}(\mathscr X)(\mathbf H,\mathbf H')$ to be  $$\mathcal S_p(\mathbf {H}, \mathbf {H}') := \{ \mbox{$\theta \in \mathbf{Hilb}(\mathscr X)(\mathbf {H}, \mathbf {H}')$ $\vert$ $\theta_c\in 
K(H(c),H'(c))$ for all $c\in \mathscr X$ and $sup_{c \in Irr(\mathscr X)} ||\theta_c||_p < \infty$}\}$$ 
\end{defn}
If $\mathbf H = \mathbf H'$, we will write $\mathcal S_p(\mathbf H)$ for $\mathcal S_p(\mathbf H, \mathbf H)$.
\begin{rem}
   The semisimplicity of $\mathscr X$ implies that if $\theta \in \mathcal{S}_p(\mathbf H, \mathbf H')$, then for every $c \in \mathscr X$, $\theta_c$ is a finite direct sum of $p$-Schatten operators and hence is $p$-Schatten. 
\end{rem}
\begin{thm}\label{P1}
 Let $\mathbf {H}, \mathbf {H}', \mathbf {H}'' \in \mathbf{Hilb}(\mathscr X)$.

\smallskip
(1) $\mathcal S_p(\mathbf {H}, \mathbf {H}')$ is a subspace of $\mathbf{Hilb}(\mathscr X)(\mathbf {H}, \mathbf {H}')$.\\
(2)  Let $\theta \in \mathcal S_p(\mathbf {H}, \mathbf {H}')$, $\theta^{'} \in \mathbf{Hilb}(\mathscr X)(\mathbf {H}', \mathbf {H}'')$ and $\theta^{''} \in \mathbf{Hilb}(\mathscr X)(\mathbf {H}'', \mathbf {H})$. Then, $\theta^{'} \circ \theta \in \mathcal S_p(\mathbf {H}, \mathbf {H}'')$ and $\theta \circ \theta^{''} \in \mathcal S_p(\mathbf {H}'', \mathbf {H}')$.\\
(3) $\mathcal S_p(\mathbf H) \subseteq \mathcal S_q(\mathbf H)$ for all $1 \leq p \leq q$.\\
(4) Let $p, q, r \geq 1$ such that $1/r = 1/p + 1/q$. Then, for $\theta \in \mathcal{S}_p(\mathbf H, \mathbf H'), \theta' \in \mathcal{S}_q(\mathbf H', \mathbf H'')$, $\theta' \circ \theta \in \mathcal S_r(\mathbf H, \mathbf H'')$.
\end{thm}
\begin{proof}
(1) Let $\theta_1, \theta_2 \in \mathcal S_p(\mathbf {H}, \mathbf {H}')$. For any object $c \in Irr(\mathscr X)$, since $(\theta_1)_c, (\theta_2)_c : \mathbf {H}(c) \longrightarrow \mathbf {H}'(c)$ are $p$-Schatten operators, hence $||(\theta_1)_c + (\theta_2)_c||_p \leq 2^{1/p}||(\theta_1)_c||_p + 2^{1/p}||(\theta_2)_c||_p$ (see for instance, \cite[$\S$ 2]{Simon}). Therefore,
\begin{equation}
    \underset{c \in Irr(\mathscr X)}{sup}||(\theta_1 + \theta_2)_c||_p \leq 2^{1/p}\underset{c \in Irr(\mathscr X)}{sup}||(\theta_1)_c||_p + 2^{1/p}\underset{c \in Irr(\mathscr X)}{sup}||(\theta_2)_c||_p < \infty.
\end{equation} This shows that $\theta_1 + \theta_2 \in \mathcal S_p(\mathbf {H}, \mathbf {H}')$. It is easy to see that $\mathcal S_p(\mathbf {H}, \mathbf {H}')$ is closed under scalar multiplication.

\smallskip
(2)  Since $\theta'\in \mathbf{Hilb}(\mathscr X)(\mathbf H',\mathbf H'')$, it follows by \eqref{26f} that $\underset{c \in \mathscr X}{sup} |\theta^{'}_{c}|<\infty$. Since $\theta\in \mathcal S_p(\mathbf H,\mathbf H')$, it follows from  Definition \ref{D3.1} that $\underset{c \in Irr(\mathscr X)}{sup} ||\theta_{c}||_p<\infty$.  For each $c\in Irr(\mathscr X)$, we know that 
$||\theta'_c\circ \theta_c||_p\leq |\theta'_c|\cdot ||\theta_c||_p$. Hence, $\underset{c \in Irr(\mathscr X)}{sup} ||(\theta' \circ \theta)_c||_p \leq \underset{c \in Irr(\mathscr X)}{sup} |\theta^{'}_{c}|.\underset{c \in Irr(\mathscr X)}{sup} ||\theta_c||_p < \infty$. The proof is similar for $\theta \circ \theta^{''}$.

\smallskip
(3) Let $\theta \in \mathcal S_p(\mathbf H)$. Since $q \geq p$, $||\theta_c||_q \leq ||\theta_c||_p$ for all $c \in Irr(\mathscr X)$ (see for instance \cite[$\S$ 2]{Simon}. Hence, $\underset{c \in Irr(\mathscr X)}{sup} ||\theta_c||_q \leq \underset{c \in Irr(\mathscr X)}{sup} ||\theta_c||_p < \infty$. It follows that $\theta \in \mathcal S_q(\mathbf H)$.

\smallskip
(4) We note that for every $c \in \mathscr X$, $(\theta' \circ \theta)_c = \theta'_c \circ \theta_c \in K(\mathbf H(c), \mathbf H''(c))$. Further, for every $c \in Irr(\mathscr X)$, $\theta_c \in \mathcal{S}_p(\mathbf H(c), \mathbf H'(c))$ and $\theta'_c \in \mathcal{S}_q(\mathbf H'(c), \mathbf H''(c))$. It follows from \cite[Appendix 1]{Con} that $\theta'_c \circ \theta_c \in \mathcal{S}_r(\mathbf H(c), \mathbf H''(c))$ and $||\theta'_c \circ \theta_c||_r \leq ||\theta'_c||_q.||\theta_c||_p$. Hence,
\begin{equation}
  \underset{c \in Irr(\mathscr X)}{sup}||\theta'_c \circ \theta_c||_r \leq \underset{c \in Irr(\mathscr X)}{sup}||\theta'_c||_q.\underset{c \in Irr(\mathscr X)}{sup}||\theta_c||_p < \infty
\end{equation}
This proves the result.
\end{proof}
\begin{cor}\label{C3.4}
For $\mathbf {H} \in \mathbf{Hilb}(\mathscr X)$, $\mathcal S_p(\mathbf H)$ is an ideal of $\mathbf{Hilb}(\mathscr X)(\mathbf {H}, \mathbf {H})$.
\end{cor}
\begin{proof}
  The result is clear by setting $\mathbf {H} = \mathbf {H}' = \mathbf {H}''$ in Proposition \ref{P1}. 
\end{proof}
We are now ready to define $p$-summable Fredholm modules in terms of actions of algebras on objects of $\mathbf{Hilb}(\mathscr X)$.
\begin{defn}  Let $A$ be an algebra over $\mathbb C$ and $p\geq 1$. 
  A $p$-summable Fredholm module over $A$ is a pair $(\mathbf {H}, \mathbf {F})$ which consists of the following data

\smallskip
(1) A $\mathbb{Z}_2$-graded object  $\mathbf  H = \mathbf {H}^+ \oplus \mathbf {H}^-$  in $\mathbf{Hilb}(\mathscr X)$ with grading operator $\varepsilon : \mathbf  H \longrightarrow \mathbf  H$, where $\varepsilon_a(h) = (-1)^{deg h}h$ for all $a \in Ob(\mathscr X)$, $h \in \mathbf {H}^{\pm}(a)$. 

\smallskip
(2) A morphism of graded algebras $\rho : A \longrightarrow \mathbf{Hilb}(\mathscr X)(\mathbf  H, \mathbf H)$ that is of degree zero. Here, $A$ is treated as a $\mathbb Z_2$-graded algebra that is concentrated in degree zero. 

\smallskip
(3) An element $\mathbf  F \in \mathbf{Hilb}(\mathscr X)(\mathbf  H, \mathbf H)$ such that

\smallskip
(i) $\mathbf {F} \circ \mathbf {F} = id$. 

\smallskip
(ii) $\mathbf {F} \circ \varepsilon = - \varepsilon \circ \mathbf {F}$.

\smallskip
(iii) $\mathbf {F} \circ \rho(a) - \rho(a) \circ \mathbf {F} \in \mathcal S_p(\mathbf H, \mathbf H)$ for all $a \in A$.
\end{defn}

\section{Characters of Fredholm modules}
For the rest of the paper, we shall assume that $Irr(\mathscr X)$ is finite.

\smallskip
Let $(\Omega, d)$ be a differential graded algebra where $\Omega = \underset{j \geq 0}{\bigoplus} \Omega^j$ is a graded algebra over $\mathbb C$ together with a graded derivation $d$ of degree $1$ such that $d^2 = 0$. We recall (see \cite[$\S$ 1]{Con}) that a closed graded trace of dimension $n \geq 0$ on $(\Omega, d)$ is a linear map $\int : \Omega^n \longrightarrow \mathbb C$ such that
\begin{equation}
    \int d\omega = 0\text{ }\text{ for all }\omega \in \Omega^{n-1}\quad\text{and}\quad\int \omega_1\omega_2 = (-1)^{ij}\int \omega_2\omega_1\text{ }\text{ for all }\omega_1 \in \Omega^i, \omega_2 \in \Omega^j\text{ with }i,j \geq 0, i+j = n
\end{equation}
  The tuple $(\Omega, d, \int)$ is called a cycle of dimension $n$. Further, if $A$ is an algebra over $\mathbb C$, a cycle of dimension $n$ on $A$ is a tuple $(\Omega, d, \int, \rho)$ where $(\Omega, d, \int)$ is a cycle of dimension $n$ and $\rho : A \longrightarrow \Omega^0$ is a morphism of algebras.

\medskip
We now fix a trivially $\mathbb Z_2$-graded algebra $A$ over $\mathbb C$ and a $p$-summable Fredholm module $(\mathbf H, \mathbf F)$ over $A$ for some $p \geq 1$. In this section, we shall associate to $(\mathbf H, \mathbf F)$ an $n$-dimensional cycle over $A$, where $n = 2m$ is any even integer such that $n \geq p-1$.

\smallskip
Let $\widetilde{A} := A \oplus \mathbb C$ be the unital algebra obtained from $A$ by adjoining a unit. The action $\rho$ of $A$ on $\mathbf H$ extends uniquely to an action of the unital algebra $\widetilde{A}$ on $\mathbf H$, which we continue to denote by $\rho$. For any $\theta \in \mathbf{Hilb}(\mathscr X)(\mathbf H, \mathbf H)$, we set $d\theta = i[F, \theta]$, where the commutator is graded. In particular, if $\theta$ is homogeneous,
\begin{equation}
    [F, \theta] = F\theta - (-1)^{deg \theta}\theta F
\end{equation}
For any $j \in \mathbb N$, we set $\Omega^j$ to be the linear span in $\mathbf{Hilb}(\mathscr X)(\mathbf H, \mathbf H)$ of elements of the form
\begin{equation}
    \rho(a_0) d\rho(a_1) d\rho(a_2) ... d\rho(a_j),\quad a_j \in \widetilde{A}
\end{equation}
\begin{lem}\label{Lemm4.1}
   (1) $d^2\theta = 0$ for all $\theta \in \mathbf{Hilb}(\mathscr X)(\mathbf H, \mathbf H)$.

   \smallskip
   (2) $d(\theta_1\theta_2) = (d\theta_1)\theta_2 + (-1)^{deg \theta_1}\theta_1(d\theta_2)$ for all $\theta_1, \theta_2 \in \mathbf{Hilb}(\mathscr X)(\mathbf H, \mathbf H)$.

   \smallskip
   (3) $d\Omega^j \subseteq \Omega^{j+1}$.

   \smallskip
   (4) $\Omega^j\Omega^k \subseteq \Omega^{j+k}$ for each $j, k \in \mathbb N$.
\end{lem}
\begin{proof}
  (1) If $\theta$ is homogeneous,
  \begin{equation}
    \begin{split}
      d^2\theta &= i[\mathbf F, d\theta] = i[\mathbf F, i[\mathbf F, \theta]] = i[\mathbf F, i\left(\mathbf F\theta - (-1)^{deg \theta}\theta \mathbf F\right)]\\
      &= -\left(\mathbf F(\mathbf F\theta - (-1)^{deg \theta}\theta \mathbf F) - (-1)^{deg \theta + 1}(\mathbf F \theta - (-1)^{\deg \theta}\theta \mathbf F)\mathbf F\right)\\
      &= -\left(\theta - (-1)^{deg \theta}\mathbf F\theta\mathbf F - (-1)^{deg \theta + 1}\mathbf F\theta\mathbf F - \theta\right) = 0.
    \end{split}
  \end{equation}
  Hence, $d^2\theta = 0$ for all $\theta \in \mathbf{Hilb}(\mathscr X)(\mathbf H, \mathbf H)$.

  \smallskip
  (2) It suffices to assume that $\theta_1, \theta_2$ are homogeneous. Then,
  \begin{equation}
    \begin{split}
        d(\theta_1\theta_2) &= i[\mathbf F, \theta_1\theta_2] = i\left(\mathbf F\theta_1\theta_2 - (-1)^{deg \theta_1 + deg \theta_2}\theta_1\theta_2\mathbf F\right)\\
        &= i(\mathbf F\theta_1 - (-1)^{deg \theta_1}\theta_1\mathbf F)\theta_2 + i\left((-1)^{deg \theta_1}\theta_1\mathbf F\theta_2 - (-1)^{deg \theta_1+\deg \theta_2}\theta_1\theta_2\mathbf F\right)\\
        &= (d\theta_1)\theta_2 + (-1)^{deg \theta_1}\theta_1(d\theta_2).
    \end{split}
  \end{equation}
  (3) Using (1) and (2), the result is clear.

  \smallskip
  (4) It is enough to show that for $a_0, a_1, ..., a_j, a \in \widetilde{A}$, $\left(\rho(a_0)d\rho(a_1)..d\rho(a_j)\right)\rho(a) \in \Omega^j$. Since $A$ is trivially $\mathbb Z_2$-graded and $\rho : A \longrightarrow \mathbf{Hilb}(\mathscr X)(\mathbf H, \mathbf H)$ is of degree $0$, we have $\left(d\rho(a_j)\right)\rho(a) = d(\rho(a_j)\rho(a)) - \rho(a_j)d\rho(a)$. The result now follows using an induction argument.
\end{proof}
Using Lemma \ref{Lemm4.1}, we see that
\begin{equation}\label{difgradal}
    \left(\Omega := \underset{0 \leq j}{\bigoplus} \Omega^j, d\right)
\end{equation}has the structure of a differential graded algebra.

\smallskip
\begin{lem}\label{lemm4.2}
     $\Omega^j \subseteq \mathcal{S}_{\frac{n+1}{j}}(\mathbf H)$ for all $1 \leq j \leq n+1$.
\end{lem}
\begin{proof}
   It suffices to show that for $a_0, a_1, ..., a_j \in \widetilde{A}$, $\rho(a_0)d\rho(a_1)...d\rho(a_j) \in \mathcal{S}_{\frac{n+1}{j}}(\mathbf H)$. Since $(\mathbf H, \mathbf F)$ is $p$-summable and $n+1 \geq p$, hence $(\mathbf H, \mathbf F)$ is $n+1$-summable. We note that for any $1 \leq k \leq j$,
  \begin{equation}
    d(\rho(a_k)) = i\left(\mathbf F\rho(a_k) - (-1)^{deg \rho(a_k)}\rho(a_k)\mathbf F\right) = i\left(\mathbf F\rho(a_k) - \rho(a_k)\mathbf F\right) \in \mathcal{S}_{n+1}(\mathbf H).
  \end{equation}
  Using Propostion \ref{P1}, it follows that 
  \begin{equation}
      d\rho(a_1)...d\rho(a_j) \in \mathcal{S}_{\frac{n+1}{j}}(\mathbf H)
  \end{equation}
  Since $\mathcal S_{\frac{n+1}{j}}(\mathbf H)$ is an ideal of $\mathbf{Hilb}(\mathscr X)(\mathbf H, \mathbf H)$, $\rho(a_0)d\rho(a_1)...d\rho(a_j) \in \mathcal{S}_{\frac{n+1}{j}}(\mathbf H)$. This completes the proof.
\end{proof}
We now fix a notation. For $\theta \in \mathcal S_1(\mathbf H)$, we set $Trace(\theta) := \sum_{c \in Irr(\mathscr X)} trace(\theta_c)$.
\begin{rem}\label{propTrace}
(1) Using \cite[$\S$ 3, Theorem 3.1]{Simon}, it follows that for $\theta_1 \in \mathcal S_1(\mathbf H), \theta_2 \in \mathbf{Hilb}(\mathscr X)(\mathbf H, \mathbf H)$
\begin{equation}
    Trace(\theta_1\theta_2) = \sum_{c \in Irr(\mathscr X)} trace\left(\left(\theta_1\right)_c\left(\theta_2\right)_c\right) = \sum_{c \in Irr(\mathscr X)} trace\left(\left(\theta_2\right)_c\left(\theta_1\right)_c\right) = Trace(\theta_2\theta_1)
\end{equation}
(2) Since the ordinary trace of operators is invariant under similarity, hence for $\theta \in \mathcal{S}_1(\mathbf H)$ and an automorphism $\zeta \in \mathbf{Hilb}(\mathscr X)(\mathbf H, \mathbf H)$,
\begin{equation}
 Trace(\zeta\theta\zeta^{-1}) = \sum_{c \in Irr(\mathscr X)}trace(\zeta_c\theta_c\zeta_c^{-1}) = \sum_{c \in Irr(\mathscr X)}trace(\theta_c) = Trace(\theta).
\end{equation}
\end{rem}
\begin{lem}\label{tracelem}
  For any $\theta \in \mathbf{Hilb}(\mathscr X)(\mathbf H, \mathbf H)$ such that $[\mathbf F, \theta] \in \mathcal{S}_1(\mathbf H)$, let
  \begin{equation}
      Tr_s(\theta) = \frac{1}{2} Trace(\varepsilon\mathbf F[\mathbf F, \theta]).
  \end{equation}
  The following hold :

  \smallskip
  (1) If $\theta$ is homogeneous with odd degree, then $Tr_s(\theta) = 0$.

  \smallskip
  (2) If $\theta \in \mathcal S_1(\mathbf H)$, then $Tr_s(\theta) = Trace(\varepsilon\theta)$.

  \smallskip
  (3) $[\mathbf F, \Omega^n] \subseteq \mathcal S_1(\mathbf H)$.

  \smallskip
  (4) $Tr_s\big{\vert}_{\Omega^n} : \Omega^n \longrightarrow \mathbb C$ is a closed graded trace of dimension $n$ on the differential graded algebra $\Omega$.
\end{lem}
\begin{proof}
    (1) Since $\varepsilon \mathbf F = - \mathbf F \varepsilon$ and $\varepsilon \theta = - \theta \varepsilon$, it follows that $\varepsilon \mathbf F[\mathbf F, \theta] = - \mathbf F[\mathbf F, \theta]\varepsilon$. Now,
    \begin{equation}
        Trace\left(\varepsilon\mathbf F[\mathbf F, \theta]\right) = Trace\left(\mathbf F[\mathbf F, \theta]\varepsilon\right) = -Trace\left(\varepsilon\mathbf F[\mathbf F, \theta]\right)
    \end{equation}
    Hence, $Trace\left(\varepsilon \mathbf F[\mathbf F, \theta]\right) = 0$ and $Tr_s(\theta) = 0$.

    \smallskip
    (2) If $\theta \in \mathcal S_1(\mathbf H)$,
    \begin{equation}\label{equat4.10}
        Trace(\varepsilon \mathbf F\theta\mathbf F) = - Trace(\mathbf F\varepsilon\theta\mathbf F) = - Trace(\varepsilon \theta)
    \end{equation}It suffices to assume that $\theta$ is homogeneous.
    If $\theta$ is homogeneous of odd degree,
    \begin{equation}
        Trace(\varepsilon\theta) = -Trace(\theta\varepsilon) = -Trace(\varepsilon\theta)
    \end{equation}
    Hence, using (1), it follows that $Trace(\varepsilon\theta) = 0 = Tr_s(\theta)$. If $\theta$ is homogeneous of even degree,
    it follows that
    \begin{equation}
        Tr_s(\theta) = \frac{1}{2}Trace\left(\varepsilon\mathbf F[\mathbf F, \theta]\right) = \frac{1}{2}Trace\left(\varepsilon \theta - \varepsilon\mathbf F\theta\mathbf F\right)
    \end{equation}
    Using $\eqref{equat4.10}$, it follows that $Tr_s(\theta) = Trace(\varepsilon\theta)$.

    \smallskip
    (3) Using Lemma \ref{lemm4.2}, it is clear that $[\mathbf F, \Omega^n] \subseteq \Omega^{n+1} \subseteq \mathcal S_1(\mathbf H)$.

    \smallskip
    (4) Since $d^2 = 0$, it follows that $Tr_s(d\theta) = 0$ for all $\theta \in \Omega^{n-1}$. We now show that for $\theta_1 \in \Omega^{n_1}, \theta_2 \in \Omega^{n_2}$ where $n_1 + n_2 = n$, we have 
    \begin{equation}\label{equat4.13}
        Tr_s(\theta_1\theta_2) = (-1)^{n_1n_2}Tr_s(\theta_2\theta_1) 
    \end{equation}Since $n$ is even, $n_1$ and $n_2$ are of the same parity and hence, $\eqref{equat4.13}$ is equivalent to 
    \begin{equation}
        Trace(\varepsilon\mathbf Fd(\theta_1\theta_2)) = (-1)^{n_1}Trace(\varepsilon \mathbf Fd(\theta_2\theta_1))
    \end{equation}
    Let $\theta \in \mathbf{Hilb}(\mathscr X)(\mathbf H, \mathbf H)$ be a homogeneous element. We note that $\theta$ and $d\theta$ are of opposite degrees and $\mathbf F. d\theta = (-1)^{deg\left(d\theta\right)} d\theta. \mathbf F$. Then,
\begin{equation}
\varepsilon\mathbf F.d\theta = (-1)^{deg\left(d\theta\right)}\varepsilon.(d\theta.\mathbf F) = (-1)^{2 deg\left(d\theta\right) + 1}d\theta.\mathbf F\varepsilon = - d\theta.\mathbf F\varepsilon = d\theta.\varepsilon\mathbf F
\end{equation}
Hence, $\varepsilon \mathbf F$ commutes with $d\theta_1$ and $d\theta_2$. Now, since $n_1 + n_2 =n$ is even, we have
\begin{equation}
 \begin{split}
    Trace(\varepsilon \mathbf Fd(\theta_1\theta_2)) &= Trace(\varepsilon \mathbf F(d\theta_1)\theta_2) + (-1)^{n_1}Trace(\varepsilon\mathbf F\theta_1(d\theta_2))\\
    &= Trace((d\theta_1)\varepsilon\mathbf F\theta_2) + (-1)^{n_1}Trace((d\theta_2)\varepsilon\mathbf F\theta_1)\\
    &= Trace(\varepsilon\mathbf F\theta_2(d\theta_1)) + (-1)^{n_1}Trace(\varepsilon\mathbf F(d\theta_2)\theta_1)\\
    &= (-1)^{n_1}Trace(\varepsilon\mathbf Fd(\theta_2\theta_1))
 \end{split}
\end{equation}
Hence, $Tr_s\big{\vert}{\Omega^n} : \Omega^n \longrightarrow \mathbb C$ is a closed graded trace of dimension $n$ on $\Omega$.
\end{proof}
We shall now associate to $(\mathbf H, \mathbf F)$, an $(n = 2m)$-dimensional cycle over $A$.
\begin{defn}
    The cycle associated to the $n+1$-summable Fredholm module $(\mathbf H, \mathbf F)$ over $A$ is the tuple $(\Omega, d, \int, \rho)$ where $(\Omega, d)$ is the differential graded algebra $\eqref{difgradal}$,
    \begin{equation}\label{equat4.19}
        \int \omega := (2i\pi)^m m!\text{ }Tr_s(w)\quad\text{for all }\omega \in \Omega^n
    \end{equation}
    is the closed graded trace of dimension $n$ on $\Omega$ and $\rho : A \longrightarrow \Omega^0 \subseteq \mathbf{Hilb}(\mathscr X)(\mathbf H, \mathbf H)$ is the homomorphism associated to the Fredholm module $(\mathbf H, \mathbf F)$.
\end{defn}
We note that the character (see \cite[$\S$ 2]{Con}) of this $n$-dimensional cycle $(\Omega, d, \int, \rho)$ is the linear map
\begin{equation}\label{char}
     \tau^n : A^{\otimes(n+1)} \longrightarrow \mathbb C,\quad
           a_0 \otimes a_1 \otimes ... \otimes a_n \mapsto (2i\pi)^m m!\text{ }Tr_s(\rho(a_0)d\rho(a_1)d\rho(a_2)...d\rho(a_n))
\end{equation}
Let $R$ be an algebra over $\mathbb C$ and let $C^k(R):=\mathbf{Vec}(R^{\otimes k+1},\mathbb C)$ for $k\geq 0$. Let $\lambda_k$ be the (signed) cyclic operator
\begin{equation}\label{uns448}
\lambda_n:C^k(R)\longrightarrow C^k(R)\qquad \lambda_k(\psi)(x_0,...,x_k):=(-1)^k\psi(x_1,...,x_k,x_0)
\end{equation}We recall from \cite[$\S$ II]{Con}, the definitions of the operators $B_0$ and $B$.
      \begin{equation}\label{newops}
        \begin{split}
        B_0 &: C^{k+1}(A) \longrightarrow C^k(A),\quad (B_0(\psi))(a_0 \otimes ... \otimes a_k) = \psi(1 \otimes a_0 \otimes ... \otimes a_k) - \psi(a_0 \otimes ... \otimes a_k \otimes 1)\\
        B &: C^{k+1}(A) \longrightarrow C^k(A), \quad B = (1+\lambda_k+...+\lambda_k^k) \circ B_0
        \end{split}
      \end{equation}
Now, let $C^k_\lambda(R):=\{\mbox{$\psi\in C^k(R)$ $\vert$ $(1-\lambda_k)(\psi)=0$}\}$. Then, the cyclic cohomology (see \cite{Con}) of $R$ can be computed by means of the complex $(C^\bullet_\lambda(R),b)$, whose differentials $b^k:C^k_\lambda(R)\longrightarrow C^{k+1}_{\lambda}(R)$ are given by
\begin{equation}\label{cobound}
(b^k\psi)(x_0,...,x_{k+1}):= \sum_{i = 0}^{k}(-1)^i \psi(x_0, ..., x_ix_{i+1}, ..., x_{k+1}) + (-1)^{k+1}\psi(x_{k+1}x_0, ..., x_k)
\end{equation} The cocycles of this complex will be denoted by $Z^\bullet_\lambda(R)$ and coboundaries by $B^\bullet_\lambda(R)$. The cohomology $Z^\bullet_\lambda(R)\big{/}B^\bullet_\lambda(R)$ of this complex will be denoted by $H^\bullet_\lambda(R)$.

\smallskip
Since $\tau^n$, as defined in $\eqref{char}$, is the character associated to the $n$-dimensional cycle $(\Omega, d, \int, \rho)$ on $A$, using \cite[Proposition 1, $\S$ II]{Con}, it immediately follows that $\tau^n \in Z^n_{\lambda}(A) \subseteq C^{n}_{\lambda}(A)$.

\begin{defn}
    The cohomology class $[\tau^n] \in H^n(A)$ of the $n$-dimensional character $\tau^n$ associated to the $p$-summable Fredholm module $(\mathbf H, \mathbf F)$ is called the $n$-dimensional Chern character $ch^n(\mathbf H, \mathbf F)$ of $(\mathbf H, \mathbf F)$.
\end{defn}

\section{Periodicity of the Chern Character for Fredholm modules}
Let $\sigma \in Z^2_{\lambda}(\mathbb C)$ be the cyclic cocycle determined by $\sigma(1, 1, 1) = 2i\rho$. We recall (see \cite[$\S$ II]{Con}) that for any $r \geq 0$, the Periodicity operator
\begin{equation}
  S : Z^r_{\lambda}(A) \longrightarrow Z^{r+2}_{\lambda}(A),\quad \psi \mapsto \psi \# \sigma
\end{equation}
 where $\#$ denotes the cup product in cyclic cohomology, takes $B^{r}_{\lambda}(A)$ into $B^{r+2}_{\lambda}(A)$. We shall continue to denote the induced map $H^r_{\lambda}(A) \longrightarrow H^{r+2}_{\lambda}(A)$ by $S$.

 \smallskip
 Let $(\mathbf H, \mathbf F)$ be a $p$-summable Fredholm module over $A$ and let $n = 2m \geq p-1$ be an even integer. Using proposition $\ref{P1}$, it follows that $(\mathbf H, \mathbf F)$ is both $n+1$ and $n+3$ summable. We consider the associated characters $\tau^n$ and $\tau^{n+2}$ of dimensions $n$ and $n+2$ respectively and show that they are related by the periodicity operator $S$.
 \begin{Thm}\label{periodicityofch}
  $ch^{n+2}(\mathbf H, \mathbf F) = S\left(ch^{n}(\mathbf H, \mathbf F)\right)$ in $H^{n+2}_{\lambda}(A)$.
 \end{Thm}
 \begin{proof}
  Let $a_0, a_1 ,..., a_{n+2} \in A$. Since $\tau^n$ is the character of the $n$-dimensional cycle $(\Omega, d, \int)$ associated to the $n+1$-summable Fredholm module $(\mathbf H, \mathbf F)$, using \cite[$\S$ II]{Con} and $\eqref{equat4.19}$, we have
     \begin{equation}\label{equat5.2}
     \begin{split}
      S\tau^{n}\left(a_0 \otimes a_1 \otimes ... \otimes a_{n+2}\right) &= 2i\pi \sum_{j = 0}^{n+1} \int\left(\rho(a_0) d\rho(a_1) d\rho(a_2) ... d\rho(a_{j-1})\right)\rho(a_j) \rho(a_{j+1}) \left(d\rho(a_{j+2}) ... d\rho(a_{n+2})\right)\\
      &= (2i\pi)^{m+1} m! \sum_{j = 0}^{n+1} Tr_s\left(\left(\rho(a_0) d\rho(a_1) d\rho(a_2) ... d\rho(a_{j-1})\right)\rho(a_j) \rho(a_{j+1}) \left(d\rho(a_{j+2}) ... d\rho(a_{n+2})\right)\right)
      \end{split}
     \end{equation}
 \end{proof}
 On the other hand, since $\tau^{n+2}$ is the character of the $(n+2)$-dimensional cycle $(\Omega, d, \int)$ associated to the $(n+3)$-summable Fredholm module $(\mathbf H, \mathbf F)$, we have
 \begin{equation}\label{equat5.3}
 \begin{split}
     \tau^{n+2}\left(a_0 \otimes a_1 \otimes ... \otimes a_{n+2}\right) &= (2i\rho)^{m+1}(m+1)! Tr_s\left(\rho(a_0) d\rho(a_1) d\rho(a_2) ... d\rho(a_{n+2})\right)
 \end{split}
 \end{equation}
 For any $0 \leq j \leq n+1$, we define $\phi_j \in C^{n+1}(A)$ by
 \begin{equation}\label{equat5.4}
   \phi_j(a_0 \otimes a_1 \otimes ... a_{n+1}) = Trace\left(\varepsilon \mathbf F \rho(a_j) d\rho(a_{j+1}) ... d\rho(a_{n+1}) d\rho(a_0) d\rho(a_1) ... d\rho(a_{j-1})\right)
 \end{equation}
 where the Trace is defined since $\rho(a_j) d\rho(a_{j+1}) ... d\rho(a_{n+1}) d\rho(a_0) d\rho(a_1) ... d\rho(a_{j-1}) \in \Omega^{n+1} \subseteq \mathcal S_1(\mathbf H)$ by Lemma $\ref{lemm4.2}$. It is clear that
 \begin{equation}\label{equat5.5}
  \sum_{j=0}^{n+1}(-1)^j \phi_j =: \phi \in C^{n+1}_{\lambda}(A)
 \end{equation}
 We now fix $0 \leq j \leq n+1$. Using $\eqref{cobound}$, $\eqref{equat5.4}$ and the equality $d(\rho(a)\rho(a')) = d\rho(a)a' + ad\rho(a')$ for $a, a' \in A$ we have
 \begin{equation}\label{equat5.6}
  \begin{split}
    b\phi_j\left(a_0 \otimes a_1 \otimes ... \otimes a_{n+2}\right) &= \sum_{i = 0}^{n+1} (-1)^i \phi_j\left(a_0 \otimes a_1 \otimes ... \otimes a_ia_{i+1} \otimes ... \otimes a_{n+2}\right) + \phi_j\left(a_{n+2}a_0 \otimes a_1 \otimes ... \otimes a_n\right)\\
    &= Trace\left(\varepsilon \mathbf F \left(\rho(a_{j+1}) d\rho(a_{j+2}) ... d\rho(a_{n+2}\right)\rho(a_0)\left(d\rho(a_1) ... d\rho(a_j)\right)\right)\\
    &\quad+ (-1)^{j-1}Trace\left(\varepsilon \mathbf F \rho(a_{j+1}) \left(d\rho(a_{j+2}) ... d\rho(a_{n+2}) d\rho(a_0) ... d\rho(a_{j-1})\right) \rho(a_j)\right)\\
    &\quad+ Trace\left(\varepsilon \mathbf F \rho(a_j)\left(d\rho(a_{j+1}) ... d\rho(a_{n+2})\right) \rho(a_0) \left(d\rho(a_1) ... d\rho(a_{j-1})\right)\right)
  \end{split}
 \end{equation}
 We set $\beta = \left(d\rho(a_{j+2}) ... d\rho(a_{n+2})\right)\rho(a_0)\left(d\rho(a_1) ... d\rho(a_{j-1})\right) \in \Omega^{n-j+1}\Omega^{j-1} \subseteq \Omega^{n}$. Using Lemma $\ref{tracelem}$, the fact that $Tr_s$ is a closed graded trace and the fact that $d\beta \in \Omega^{n+1} \subseteq \mathcal S_1(\mathbf H)$, we have
 \begin{equation}\label{equat5.7}
  Trace(\varepsilon \alpha (d\beta)) = Tr_s(\alpha (d\beta)) = Tr_s((d\alpha) \beta)
 \end{equation}
for all homogeneous elements $\alpha \in \mathbf{Hilb}(\mathscr X)(\mathbf H, \mathbf H)$ of odd degree. In particular for $\alpha = \rho(a_j) \mathbf F \rho(a_{j+1})$, we have
\begin{equation}
  \begin{split}
     Tr_s(i[\mathbf F, \rho(a_j) \mathbf F \rho(a_{j+1})]\beta) &= Trace(\varepsilon (\rho(a_j) \mathbf F \rho(a_{j+1}))(d\beta)) = Trace(\rho(a_j) \varepsilon \mathbf F \rho(a_{j+1}) (d\beta)) = Trace(\varepsilon \mathbf F \rho(a_{j+1}) (d\beta)\rho(a_j))\\
    &= (-1)^{j-1} Trace(\varepsilon \mathbf F \rho(a_{j+1}) d\rho(a_{j+2})...d\rho(a_{n+2})d\rho(a_0) d\rho(a_1) ... d\rho(a_{j-1}) \rho(a_j))
  \end{split}
\end{equation}
Using $\eqref{equat5.6}$, we now have
\begin{equation}
  \begin{split}
  b\phi_j\left(a_0 \otimes a_1 \otimes ... \otimes a_{n+2}\right) &= Trace(da_j\varepsilon\mathbf F a_{j+1}\beta) + Tr_s(i[\mathbf F, \rho(a_j)\mathbf F\rho(a_{j+1})]\beta) + Trace(\varepsilon \mathbf F \rho(a_j) d\rho(a_{j+1})\beta)\\
  &= Tr_s((\mathbf F d(\rho(a_j)\rho(a_{j+1})) + i[\mathbf F, \rho(a_j) \mathbf F \rho(a_{j+1})])\beta)\\
  &= -i\text{ }Tr_s((d\rho(a_j)d\rho(a_{j+1}) - 2\rho(a_j)\rho(a_{j+1}))\beta)\\
  &= \frac{(-1)^{j-1}}{i} Tr_s(\rho(a_0) d\rho(a_1) ... d\rho(a_{n+2}))\\
  &\qquad + \frac{2}{i}(-1)^{j} Tr_s((\rho(a_0)d\rho(a_1) ... d\rho(a_{j-1}))\rho(a_j)\rho(a_{j+1})(d\rho(a_{j+2}) ... d\rho(a_{n+2})))
  \end{split}
\end{equation}
Finally, using $\eqref{equat5.5}$, it follows that
\begin{equation}
\begin{split}
b\phi(a_0 \otimes a_1 \otimes ... \otimes a_{n+2}) &= \frac{2}{i} \sum_{j = 0}^{n+1} Tr_s((\rho(a_0)d\rho(a_1)...d\rho(a_{j-1}))\rho(a_j)\rho(a_{j+1})(d\rho(a_{j+2})...d\rho(a_{n+2})))\\
&\quad - \frac{n+2}{i} Tr_s(\rho(a_0)d\rho(a_1)...d\rho(a_{n+2})) 
\end{split}
\end{equation}Since $a_0, ..., a_{n+2} \in A$ are arbitrary, using $\eqref{equat5.2}$ and $\eqref{equat5.3}$, $b\left(2^m i^{m+2} \pi^{m+1} m! \phi\right) = S\tau^n - \tau^{n+2}$. This proves the result.

\section{Homotopy invariance of the Chern Character}
For any $\mathbf H \in \mathbf{Hilb}(\mathscr X)$, we note that the function $||.|| : \mathbf{Hilb}(\mathscr X)(\mathbf H, \mathbf H) \longrightarrow \mathbb R_{\geq 0},\text{ }\theta \mapsto \underset{c \in \mathscr X}{sup}||\theta_c||$
makes $\mathbf{Hilb}(\mathscr X)(\mathbf H, \mathbf H)$ into a normed algebra. Further, for any $p \geq 1$, since $||.||_p$ is a norm on the algebra of $p$-Schatten operators on a hilbert space (see for instance \cite[Appendix I]{Con}), the function $\mathcal{S}_p(\mathbf H) \longrightarrow \mathbb R_{\geq 0}, \text{ }\theta \mapsto \underset{c \in Irr(\mathscr X)}{sup}||\theta_c||_p$ makes $\mathcal{S}_p(\mathbf H)$ into a normed algebra.
\begin{rem}
We note that for each $c \in Irr(\mathscr X)$, the evaluation map $\mathcal{S}_1(\mathbf H) \longrightarrow S_1(\mathbf H(c)), \theta \mapsto \theta_c$ is continuous. Hence, $Trace : \mathcal{S}_1(\mathbf H) \longrightarrow \mathbb C,\text{ }\theta \mapsto \sum_{c \in Irr(\mathscr X)}trace(\theta_c)$, being a sum of continuous maps, is continuous.
\end{rem}
Let $p = 2m$ be an even integer, $A$ be a trivially $\mathbb Z_2$-graded algebra over $\mathbb C$ and $\mathbf H \in \mathbf{Hilb}(\mathscr X)$. Let $(\mathbf H, \mathbf F)$ and $(\mathbf H, \mathbf F')$ be two $p$-summable Fredholm modules over $A$ with the same underlying $A$-module $\mathbf H$. In this section, we will prove that if there is a homotopy $[0, 1] \ni t \mapsto (\mathbf H, \mathbf F_t)$ from $(\mathbf H, \mathbf F)$ to $(\mathbf H, \mathbf F')$, then $ch^{p+2}(\mathbf H, \mathbf F) = ch^{p+2}(\mathbf H, \mathbf F')$ in $H^{p+2}_{\lambda}(A)$.

\medskip
Let $\mathbf H_0 \in \mathbf{Hilb}(\mathscr{X})$ and let $\mathbf H$ be the $\mathbb Z_2$-graded Hilbert space object with $\mathbf H^+ = \mathbf H_0$ and $\mathbf H^- = \mathbf H_0$. We consider the natural transformation
\begin{equation}
    \mathbf F = 
    \begin{bmatrix}
        0 && id_{\mathbf H_0}\\
        id_{\mathbf H_0} && 0
    \end{bmatrix}.
\end{equation}For each $a \in \mathscr{X}$, since $||\mathbf F_a|| = 1$, we have $\sup_{a \in \mathscr X}||\mathbf F_a|| = 1 < \infty$. Hence, $\mathbf F \in \mathbf{Hilb}(\mathscr X)(\mathbf H, \mathbf H)$.
\begin{lem}\label{homtopinv}
 Let $p = 2m$ be an even integer. For each $t \in [0, 1]$, let $\rho_t : A \longrightarrow \mathbf{Hilb}(\mathscr X)(\mathbf H, \mathbf H)$ be a graded homomorphism such that

 \smallskip
 (1) For each $a \in A$, the association $t \mapsto [\mathbf F, \rho_t(a)]$ is a continuous map $\mu_a : [0, 1] \longrightarrow \mathcal{S}_p(\mathbf H)$.

 \smallskip
 (2) For each $a \in A$, the association $t \mapsto \left(\rho_t\right)(a)$ is a piecewise strongly $C^1$ map $\nu_a : [0, 1] \longrightarrow \mathbf{Hilb}(\mathscr X)(\mathbf H, \mathbf H)$.

 \smallskip
 Let $(\mathbf H_t, \mathbf F)$ be the corresponding $p$-summable Fredholm modules over $A$ where $\mathbf H_t = \mathbf H$ for all $t \in [0, 1]$. Then, the class in $H^{p+2}_{\lambda}(A)$ of the $p+2$-dimensional character of the Fredholm module $(\mathbf H_t, \mathbf F)$ is independent of $t$. 
\end{lem}
\begin{proof}
    We only prove the case in which the map $\nu_a$ is strongly $C^1$ for each $a \in A$, the proof of the general case being similar. For any $t \in [0, 1]$, we consider the $p$-dimensional character $\tau^p_t$ of $(\mathbf H_t, \mathbf F)$. We fix an $l \in [0, 1]$. We shall show that $S(\tau^p_l) = S(\tau^p_0)$ in $H^{p+2}_{\lambda}(A)$. We denote the map $t \mapsto \tau^p_t$ by $\tau^p$. Without loss of generality, we assume that $A$ is unital and $\rho_t(1) = id_{\mathbf H}$ for all $t \in [0, 1]$. For every $a \in A$, let $\delta_{\_}(a) = \nu_a' : [0, 1] \longrightarrow \mathbf{Hilb}(\mathscr X)(\mathbf H, \mathbf H)$. We note that for $a, b \in A, t \in [0, 1]$, 
    \begin{equation}
    \begin{split}
      &\text{ }\delta_t(ab) - \rho_t(a) \circ \delta_t(b) - \delta_t(a) \circ \rho_t(b)\\
      =&\text{ }\nu_{ab}'(t) - \rho_t(a)\nu_{b}'(t) - \nu_a'(t)\rho_t(b)\\
      =&\text{ }\underset{s \to 0}{lim}\frac{1}{s}\left(\nu_{ab}(t+s) - \nu_{ab}(t) - \rho_t(a) \circ \nu_b(t+s) + \rho_t(a) \circ \nu_b(t) - \nu_a(t+s) \circ \rho_t(b) + \nu_a(t) \circ \rho_t(b) \right)\\
      =&\text{ }\underset{s \to 0}{lim}\frac{1}{s}\left(\rho_{t+s}(ab) - \rho_{t}(ab) - \rho_t(a) \circ \rho_{t+s}(b) + \rho_t(a) \circ \rho_t(b) - \rho_{t+s}(a) \circ \rho_t(b) + \rho_t(a) \circ \rho_t(b) \right)\\
      =&\text{ }\underset{s \to 0}{lim}\frac{1}{s}\left(\left(\rho_{t+s}(a) - \rho_t(a)\right)  \circ\left(\rho_{t+s}(b) - \rho_t(b)\right)\right)\\
      =&\text{ }\underset{s \to 0}{lim}\frac{1}{s}\left(\left(\nu_{a}(t+s) - \nu_a(t)\right) \circ \left(\nu_b(t+s) - \nu_b(t)\right)\right)\\
      =&\text{ }\nu_a'(t) \circ \underset{s \to 0}{lim}\left(\nu_b(t+s) - \nu_b(t)\right) = 0
      \end{split}
    \end{equation}
    using the continuity of $\nu_b$. Hence, for all $a, b \in A, t \in [0, 1]$
    \begin{equation}\label{equat6.3}
        \delta_t(ab) = \rho_t(a) \circ \delta_t(b) + \delta_t(a) \circ \rho_t(b).
    \end{equation} For any $t \in [0, 1]$, we consider the $p+2$-linear functional
    \begin{equation}\label{equat6.4}
        \phi_t : A^{p+2} \longrightarrow \mathbb C,\quad a_0 \otimes a_1 \otimes ... \otimes a_{p+1} \mapsto \sum_{k = 1}^{p+1}(-1)^{k-1} Trace\left(\varepsilon\rho_t(a_0)[\mathbf F, \rho_t(a_1)]...[\mathbf F, \rho_t(a_{k-1})]\delta_t(a_k)[\mathbf F, \rho_t(a_{k+1})]...[\mathbf F, \rho_t(a_{p+1})]\right)
    \end{equation}
    Using $\eqref{equat6.3}$, it can be checked that $b(\phi_t) = 0$ for all $t \in [0, 1]$, where $b$ is the differential operator of the Hochschild cochain complex of $A$.

    \smallskip
    For a fixed $a \in A$, the compactness of $[0, 1]$ and the assumptions (1) and (2) imply that the families $\{\rho_t(a) : t \in [0, 1]\}$ and $\{\delta_t(a) : t \in [0, 1]\}$ are uniformly bounded. Hence, it makes sense to define the $p+2$-linear functional 
    \begin{equation}
      \phi : A^{p+2} \longrightarrow \mathbb C,\quad a_0 \otimes a_1 \otimes ... \otimes a_{p+1} \mapsto (2i\pi)^mm!\int_{0}^{l} \phi_t\left(a_0 \otimes a_1 \otimes ... \otimes a_{p+1}\right)
    \end{equation}
    Since $b(\phi_t) = 0$ for all $t \in [0, 1]$, $b(\phi) = 0$. Further, for $a_0, a_1, ..., a_p \in A$, $\eqref{equat6.4}$ implies that
    \begin{equation}\label{equati6.6}
      \phi(1, a_0, ..., a_p) = (2i\pi)^mm!\int_{0}^{l}\sum_{k = 0}^{p}(-1)^k Trace\left(\varepsilon[\mathbf F, \rho_t(a_1)]...[\mathbf F, \rho_t(a_{k-1})]\delta_t(a_k)[\mathbf F, \rho_t(a_{k+1})]...[\mathbf F, \rho_t(a_{p+1})]\right)
    \end{equation}
    We note that for any $a_0, a_1, ..., a_p \in A$, $t \in [0, 1]$, since $[\mathbf F, \rho_t(a_1)],...,[\mathbf F, \rho_t(a_p)] \in \mathcal{S}_p(\mathbf H)$, hence using Proposition \ref{P1}, it follows that $\rho_t(a_0)[\mathbf F, \rho_t(a_1)]...[\mathbf F, \rho_t(a_p)] \in \mathcal{S}_1(\mathbf H)$. Using Lemma \ref{tracelem}, we have
    \begin{equation}
    \tau^p_t(a_0 \otimes a_1 \otimes ... \otimes a_p) = (2i\pi)^mm!Tr_s\left(\rho_t(a_0)[\mathbf F, \rho_t(a_1)]...[\mathbf F, \rho_t(a_p)]\right) = (2i\pi)^mm!Trace\left(\varepsilon\rho_t(a_0)[\mathbf F, \rho_t(a_1)]...[\mathbf F, \rho_t(a_p)]\right)
    \end{equation}
    Since (1) implies that the map $t \mapsto [\mathbf F, \rho_t(a)]$ is continuous for each $a \in A$, it follows that for every $1 \leq k \leq p$
    \begin{equation}\label{equat6.8}
      \begin{split}
          &\underset{s \to 0}{lim}\text{ }Trace\left(\varepsilon\rho_t(a_0)[\mathbf F, \rho_t(a_1)]...[\mathbf F, \rho_t(a_{k-1})][\mathbf F, \frac{1}{s}\left(\rho_{t+s}(a_k) - \rho_t(a_k)\right)][\mathbf F, \rho_{t+s}(a_{k+1})]...[\mathbf F, \rho_{t+s}(a_p)]\right)\\
          &\text{ }= \underset{s \to 0}{lim}(-1)^k\text{ }Trace\left(\varepsilon[\mathbf F, \rho_t(a_0)]...[\mathbf F, \rho_t(a_{k-1})]\frac{1}{s}\left(\rho_{t+s}(a_k) - \rho_t(a_k)\right)[\mathbf F, \rho_{t+s}(a_{k+1})]...[\mathbf F, \rho_{t+s}(a_p)]\right)\\
          &\text{ }= (-1)^k\text{ }Trace\left(\varepsilon[\mathbf F, \rho_t(a_0)]...[\mathbf F, \rho_t(a_{k-1})]\delta_t(a_k)[\mathbf F, \rho_t(a_{k+1})]...[\mathbf F, \rho_t(a_p)]\right)
      \end{split}
      \end{equation}
      Hence, 
      \begin{equation}\label{equati6.9}
      \begin{split}
          (\tau^p)'(t)(a_0 \otimes ... \otimes a_p) &= \underset{s \to 0}{lim}\left(\tau^p(t+s)(a_0 \otimes ... \otimes a_p) - \tau^p(t)(a_0 \otimes ... \otimes a_p)\right)\\
          &= (2i\pi)^mm!\text{ }\underset{s \to 0}{lim}(Trace(\varepsilon\frac{1}{s}\left(\rho_{t+s}(a_0) - \rho_{t}(a_0)\right)[\mathbf F, \rho_{t+s}(a_1)]...[\mathbf F, \rho_{t+s}(a_p)])\\
          &\text{ }\text{ }+ Trace(\varepsilon\rho_t(a_0)[\mathbf F, \frac{1}{s}\left(\rho_{t+s}(a_1) - \rho_t(a_1)\right)][\mathbf F, \rho_{t+s}(a_2)]...[\mathbf F, \rho_{t+s}(a_p)]) + ... \\
          &\text{ }\text{ }+ Trace(\varepsilon\rho_t(a_0)[\mathbf F, \rho_t(a_1)]...[\mathbf F, \rho_t(a_{p-1})][\mathbf F, \frac{1}{s}\left(\rho_{t+s}(a_p) - \rho_t(a_p)\right)]))\\
          &= (2i\pi)^mm!\left(\sum_{k = 0}^{p} (-1)^k Trace(\varepsilon[\mathbf F, \rho_t(a_0)]...[\mathbf F, \rho_t(a_{k-1})]\delta_t(a_k)[\mathbf F, \rho_t(a_{k+1})]...[\mathbf F, \rho_t(a_p)])\right)
      \end{split}
      \end{equation}
      Now using $\eqref{equati6.6}$ and $\eqref{equati6.9}$, we have
      \begin{equation}\label{equati6.10}
          \phi(1, a_0, ..., a_p) = \int_{0}^{l}(\tau^p)'(t)(a_0 \otimes ... \otimes a_p) dt = \tau_l^p(a_0 \otimes ... \otimes a_p) - \tau_0^p(a_0 \otimes ... \otimes a_p)
      \end{equation}
      Since $b(\phi) = 0 \in C^{p+2}_{\lambda}(A)$, using $\eqref{newops}$ and $\eqref{equati6.10}$,
      \begin{equation}
        \begin{split}
       B_0(\phi)(a_0 \otimes ... \otimes a_p) &= \phi(1 \otimes a_0 \otimes ... \otimes a_p) - \phi(a_0 \otimes ... \otimes a_p \otimes 1)\\
       &= \phi(1 \otimes a_0 \otimes ... \otimes a_p) = (\tau_l^p - \tau_0^p)(a_0 \otimes ... \otimes a_p)
       \end{split}
      \end{equation}
      so that $B_0(\phi) = \tau_l^p - \tau_0^p$. Using \cite[$\S$ II, Lemma 34]{Con} and the fact that $(\tau_l^p - \tau_0^p) \in C^p_{\lambda}(A)$, it follows that $(p+1)(\tau_l^p - \tau_0^p) = (1 + \lambda + ... + \lambda^p)(\tau_l^p - \tau_0^p) = B(\phi) \in Z^p_{\lambda}(A)$ and
      \begin{equation}
         (p+1)S(\tau_l^p - \tau_0^p) = S(B(\phi)) = 2i\rho(p+1)(p+2)b(\phi) = 0\text{ in }H^{p+2}_{\lambda}(A)
      \end{equation}Finally, an application of Theorem \ref{periodicityofch} completes the proof.
\end{proof}
\begin{Thm}\label{hominvtheo}
Let $p = 2m$ be an even integer. Let $A$ be a trivially $\mathbb Z_2$-graded algebra over $\mathbb C$, $\mathbf H$ a $\mathbb Z_2$-graded object in $\mathbf{Hilb}(\mathscr X)$. Let $\{(\mathbf H, \mathbf F_t) : t \in [0, 1]\}$ be a family of $p$-summable Fredholm modules over $A$ where for each $t \in [0, 1]$,
\begin{equation}
  \mathbf F_t = \begin{bmatrix}
    0 && \mathbf Q_t\\
    \mathbf P_t && 0
  \end{bmatrix}
\end{equation}For each $t \in [0, 1]$, let $\rho_t : A \longrightarrow \mathbf{Hilb}(\mathscr X)(\mathbf H, \mathbf H)$ be the corresponding graded homomorphism of degree $0$ with its two components $\rho_t^{\pm}$. Further, assume that for every $a \in A$,

\smallskip
(1) $t \mapsto \rho_t^{+}(a) - \mathbf Q_t\rho_t^{-}(a)\mathbf P_t$ is a continuous map $[0, 1] \longrightarrow \mathcal{S}_p(\mathbf H)$.

\smallskip
(2) $t \mapsto \rho_t^{+}(a)$  and $t \mapsto \mathbf Q_t\rho_t^{-}(a)\mathbf P_t$ are piecewise strongly $C_1$ maps.

\smallskip
Then, the $(p+2)$-dimensional character $ch^{p+2}(\mathbf H, \mathbf F_t) \in H^{p+2}_{\lambda}(A)$ is independent of $t \in [0, 1]$.
\end{Thm}
\begin{proof}
  Let $a \in A$. For each $t \in [0, 1]$, since $(\mathbf H, \mathbf F_t)$ is a Fredholm module, $\mathbf F_t^2 = id$ and hence $\mathbf P_t^{-1} = \mathbf Q_t$. We set
  \begin{equation}
    \mathbf T_t = \begin{bmatrix}
     id && 0\\
     0 && \mathbf Q_t
    \end{bmatrix}
  \end{equation}It follows that $\mathbf T_t^{-1} = \begin{bmatrix}
      id && 0\\
      0 && \mathbf P_t
  \end{bmatrix}$ so that, $\mathbf T_t\mathbf F_t \mathbf T_t^{-1} = \begin{bmatrix}
      0 && id\\
      id && 0
  \end{bmatrix}$ and $\mathbf T_t\rho_t(a)\mathbf T_t^{-1} = \begin{bmatrix}
      \rho_t^{+}(a) && 0\\
      0 && \mathbf Q_t\rho_t^{-}(a)\mathbf P_t
  \end{bmatrix}$.

  \smallskip
  Using assumption (2), it follows that for each $a \in A$, the map $t \mapsto \mathbf T_t\rho_t(a)\mathbf T_t$ is piecewise strongly $C^1$.
  
  \smallskip
  We also see that
  $[\mathbf T_t \mathbf F_t \mathbf T_t^{-1}, \mathbf T_t \mathbf \rho_t(a) \mathbf T_t^{-1}] = \begin{bmatrix}
      0 && \mathbf Q_t\rho_t^{-}(a)\mathbf P_t - \rho_t^{+}(a)\\
      \rho_t^{+}(a) - \mathbf Q_t\rho_t^{-}(a)\mathbf P_t && 0
  \end{bmatrix}$.

  \smallskip
  Using assumption (1), it follows that for each $a \in A$, the map $t \mapsto [\mathbf T_t \mathbf F_t \mathbf T_t^{-1}, \mathbf T_t \mathbf \rho_t(a) \mathbf T_t^{-1}]$ is a continuous map $[0, 1] \longrightarrow \mathcal{S}_p(\mathbf H)$. It follows that the family $\{(\mathbf H, \mathbf T_t\mathbf F_t\mathbf T^{-1})\}$ of $p$-summable Fredholm modules satisfies the hypothesis of Lemma \ref{homtopinv}. The result now follows using Lemma \ref{homtopinv} and the invariance of $Trace : \mathcal{S}_1(\mathbf H) \longrightarrow \mathbb C$ under similarity (Remark \ref{propTrace}).
\end{proof}
\begin{cor}
  Let $p =2m$ be an even integer. Let $\{(\mathbf H, \mathbf F_t) : t \in [0, 1]\}$ be a family of $p$-summable Fredholm modules over $A$ with the same underlying graded homomorphism $\rho$. Further, assume that $t \mapsto \mathbf F_t$ is a strongly $C^1$ map $[0, 1] \longrightarrow \mathbf{Hilb}(\mathscr X)(\mathbf H, \mathbf H)$. Then, the $(p+2)$-dimensional character $ch^{p+2}(\mathbf H, \mathbf F_t) \in H_{\lambda}^{p+2}(A)$ of $(\mathbf H, \mathbf F_t)$ is independent of $t \in [0, 1]$.
\end{cor}
\begin{proof}
    Let $t \in  [0, 1]$. We note that since $\mathbf F_t \circ \varepsilon = - \varepsilon \circ \mathbf F_t$ and $\mathbf F_t^2 = id$, we have $\mathbf F_t = \begin{bmatrix}
        0 && \mathbf Q_t\\
        \mathbf P_t && 0
    \end{bmatrix}$ for some $\mathbf P_t \in \mathbf{Hilb}(\mathscr X)(\mathbf H^+, \mathbf H^-)$ and $\mathbf Q_t \in \mathbf{Hilb}(\mathscr X)(\mathbf H^-, \mathbf H^+)$ with $\mathbf Q_t = \mathbf P_t^{-1}$. Further, $\rho(a) = \begin{bmatrix}
      \rho^+(a) && 0\\
      0 && \rho^{-}(a)
    \end{bmatrix}$ where $\rho^{\pm}(a)$ are the components of $\rho(a)$ for every $a \in A$. It is clear that the system $\{(\mathbf H, \mathbf F_t) : t \in [0, 1]\}$ satisfies the assumptions (1) and (2) of Theorem $\ref{hominvtheo}$. This proves the result.
\end{proof}

\end{document}